\documentclass[11pt,a4paper]{article}

\usepackage[a4paper, total={5.8in, 8.5in}]{geometry}

\usepackage{amsmath,amsfonts}
\usepackage{amsthm,amssymb,hyperref}
\usepackage[english]{babel}
\usepackage{enumitem}
\usepackage{mdframed}

\usepackage{xcolor} 
\hypersetup{
    colorlinks,
    linkcolor={blue},
    citecolor={red},
    urlcolor={blue}
}

\newtheorem{theorem}{Theorem}
\newtheorem{lemma}[theorem]{Lemma}
\newtheorem{conjecture}[theorem]{Conjecture}

\newtheorem{proposition}[theorem]{Proposition}
\newtheorem{observation}[theorem]{Observation}

\numberwithin{theorem}{section}

\newcommand{\sm}{\setminus}
\newcommand{\cs}{\mathcal{S}}

\begin{document}

\title{Domination and fractional domination in digraphs}

\author{Ararat Harutyunyan$^a$, Tien-Nam Le$^b$,\\ Alantha Newman$^c$, and St\'{e}phan Thomass\'{e}$^b$\\~\\
\small $^a$LAMSADE, CNRS Universit\'{e} Paris-Dauphine
 \\ \small PSL Research University \\ \small 75016 Paris, France\\~\\
\small $^b$Laboratoire d'Informatique du Parall\'elisme \\ \small   UMR 5668 ENS Lyon - CNRS - UCBL - INRIA
\\ \small Universit\'e de Lyon, France\\~\\
\small $^c$Laboratoire G-SCOP \\ \small CNRS, Universit\'e Grenoble-Alpes \\ \small Grenoble, France}

\date{}

\maketitle

\begin{abstract}
In this paper, we investigate the relation between the (fractional)
domination number of a digraph $G$ and the independence number of its
underlying graph, denoted by $\alpha(G)$. More precisely, we prove
that every digraph $G$ on $n$ vertices has fractional domination number at most
$2\alpha(G)$ and domination number at most $2\alpha(G) \cdot \log{n}$.
Both bounds are sharp.
\end{abstract}

\section{Introduction}\label{sec:intro}

Every digraph in this paper is simple, loopless and finite, where a
digraph $G$ is \emph{simple} if for every two vertices $u$ and $v$ of
$G$, there is at most one arc with endpoints $\{u,v\}$.  Given a
digraph $G$, we denote by $V(G)$ and $E(G)$ the vertex set and arc set
of $G$, respectively. The \emph{independence number} $\alpha(G)$ of a
digraph $G$ is the independence number of the underlying (undirected)
graph of $G$. A digraph with independence number 1 is called a
\emph{tournament}.  An {\em Eulerian tournament} is a tournament that
is in addition Eulerian (i.e., the indegree of each vertex equals its
outdegree).

Given a digraph $G=(V,E)$, we say that a vertex of $G$
\emph{dominates} itself and all of its out-neighbors. A set of
vertices $S \subseteq V$ is called a \emph{dominating set} if every
vertex $v \in V$ is either an element of $S$ or is an out-neighbor of
some element of $S$.  The \emph{domination number} $γ\gamma(G)$ is the
cardinality of a minimum (by cardinality) dominating set of $G$.
Domination in tournaments has been well-studied \cite{ABK06},
\cite{GP14}, \cite{MV88}, while little is known for domination in
general digraphs. Recently, it was proved \cite{CKLST},
\cite{HLNT17+}, \cite{HLTW17+} that the domination number is closely
related to the dichromatic number when the digraph is a
tournament or a dense digraph.  
The topic of domination in undirected
graphs (where $S \subseteq V$ is a \emph{dominating set} if every
vertex $v \in V$ is either an element of $S$ or is a neighbor of some
element of $S$) has also been studied extensively, see for example the
monograph \cite{HHS}.  It is a well-known fact that in an undirected
graph, any maximal independent set is also a dominating set.

Suppose we are given a digraph $G=(V,E)$, a subset $S$ of $V$, and a
function $g:V\to\{0,1\}$ such that $g(v) =1$ if $v\in S$ and $g(v)= 0$
otherwise. Then $S$ is a dominating set of $G$ if and only if
$\sum_{x\in N^-(v)\cup\{v\}}g(x)\ge 1$ for every $v\in V$, where
$N^-(v)$ is the set of in-neighbors of $v$. Thus, a natural linear
relaxation of domination in digraphs arises as follows.  A
\emph{fractional dominating function} of $G$ is a function $g : V \to
     [0,1]$ such that $\sum_{x\in N^-(v)\cup\{v\}}g(x)\ge 1$ for every
     $v\in V$.  The \emph{fractional domination number} $\gamma^*(G)$
     is the smallest value of $\sum_{v\in V}g(v)$ over all fractional
     dominating functions $g$ of $G$.\footnote{We remark that when constructing a dominating function
  $g(\cdot)$ to upper bound $\gamma^*(G)$ by the value $g(V)$, it is
  sufficient to show that $g(v) \geq 0$ and 
$\sum_{x\in N^-(v)\cup\{v\}}g(x)\ge 1$ for every $v \in V$.  If $g(v)
  > 1$ for some $v \in V$, then the 
function $g(\cdot)$ is not minimal (i.e., $g(v)$ can be decreased).}
The fractional domination number
     of a tournament was the main tool to prove the long standing
     Erd\"os-Sands-Sauer-Woodrow conjecture in \cite{BLT}.

In this paper, we show that for any digraph, its fractional domination
number is at most twice its independence number, and this bound is
sharp.
\begin{theorem}\label{thm:frac_dom}
For every digraph $G$, we have $\gamma^*(G) \leq 2\alpha(G)$.
\end{theorem}

In contrast to the fractional domination number, it is not possible to
bound the domination number of a digraph in terms of its independence
number.  Indeed, it was shown in \cite{L01} that almost surely a
random tournament has domination number on the order of $\log n$, much
larger than its independence number of 1.  However, the upper bound of
$\log{n}$ on the domination number of a tournament can be extended to
general digraphs.

\begin{theorem}\label{thm:int_dom}
For every digraph $G$ on $n$ vertices, we have $\gamma(G) \leq \alpha(G) \cdot \log{n}$.
\end{theorem}

Sometimes, it is in fact possible to bound the domination number of a
digraph purely in terms of its independence number.  We discuss this
further in the last section.

\subsection{Notation}
Let $G=(V,{E})$ be a digraph; for every $v\in V$, we denote by
$N_G^+(v),N^-(v)$ the set of out-neighbors and in-neighbors of $v$,
respectively.  Let $N_G^+[v] = N_G^+(v) \cup\{v\} $ and
$N_G^-[v]=N_G^-(v) \cup\{v\} $.  Given a subset $S$ of $V$, we write
$N_G^+(S)= \bigcup_{v\in S}N_G^+(v)$, and similarly for $N_G^-(S),
N_G^+[S]$, and $N_G^-[S]$.  Given two vertices $u,v$, if $uv,vu\notin
E$, we say that $u$ and $v$ are \emph{independent}. We denote by
$N_G^o(v)$ the set of vertices that are independent with $v$ (i.e.,
$N_G^o(v)= V\sm(N_G^+[v]\cup N_G^-[v])$).  When it is clear from the
context, we may omit the subscript $G$. Given a subset $X$ of $V$, we
denote by $G[X]$ the induced subgraph of $G$ on $X$.  Given a digraph
$G=(V,E)$ and a function $g$ on $V$, we write $g(X):= \sum_{v\in
  X}g(v)$ for short.  We sometimes use $n$ to denote $|V(G)|$.
Finally, we mention a trivial but useful observation regarding the
independence number of a digraph.

\begin{observation}\label{ob:indep}
If $G$ is a digraph with independence number $\alpha$ and $v$ is an
arbitrary vertex in $G$, then $G[N^o(v)]$ has independence number at
most $\alpha - 1$.
\end{observation}

\section{Fractional domination in digraphs}\label{sec:frac}

In this section, we present two proofs of Theorem \ref{thm:frac_dom}.
The first proof uses the duality of linear programming, while the
second proof is by induction.  We first present some useful lemmas.
\begin{lemma}\label{lem:matrix1}
Given a digraph $G=(V,E)$ and a function $p:V\to[0,1]$, there is a
vertex $v\in V$ such that $p(N^-(v))\le p(N^+(v))$.
\end{lemma}
\begin{proof}
Suppose that the lemma is false. Then for the function $p$, we have
$p(N^-(v)) > p(N^+(v))$ for every $v \in V$ (i.e.,
$\sum_{x\in
  N^-(v)}p(x)> \sum_{y\in N^+(v)}p(y)$ for every $v$). Hence
\begin{align*}
\sum_{v\in V}p(v)\bigg(\sum_{x\in N^-(v)}p(x)\bigg)&> \sum_{v\in
  V}p(v)\bigg(\sum_{y\in N^+(v)}p(y)\bigg)\\
\Longrightarrow \ \ \ \ \ \ \ \ \ \ \  \sum_{xv\in E}p(v)p(x)&> \sum_{vy\in E}p(v)p(y),
\end{align*}
a contradiction.
\end{proof}

\begin{lemma}\label{lem:matrix2}
Given a digraph $G=(V,E)$ and a function $p:V\to [0,1]$ with $p(V) > 0$,
there is a  stable set $S\subseteq V$ such that $p(N^+[S]) \geq p(V)/2$.
\end{lemma}
\begin{proof}
We prove the lemma by induction on $|V|$. The lemma clearly holds for
$|V|=1$. For $|V|>1$, fix some function $p:V \to [0,1]$ and 
apply Lemma \ref{lem:matrix1} to obtain
a vertex $v$ such that $p(N^-(v)) \leq p(N^+(v))$.
If $N^o(v)=\emptyset$, then 
\begin{align*}
2 p(N^+[v]) & = 2p(N^+(v))+ p(v)\\
& \geq p(N^+(v))+p(N^-(v))+p(v)\\ & = p(V),
\end{align*}
 which proves the lemma.

If $N^o(v)\ne \emptyset$, we apply induction on $G[N^o(v)]$ to obtain
a stable set $S'$ such that $p(N_{G[N^o(v)]}^+[S']) \geq 
p(N^o(v))/2$. Let $S = S'\cup \{v\}$. We have the following remarks.
\begin{itemize}
\item $N_{G[N^o(v)]}^+[S'] = N^+[S']\cap N^o(v)$, and
\item $N^+[v]$ and $N^+[S']\cap N^o(v)$ are disjoint.
\end{itemize}
Thus, for the stable set $S$, we have
\begin{align*}
p(N^+[S]) &\ge p(N^+[v])+ p\big(N^+[S']\cap N^o(v)\big)\\
&= p(v) + p(N^+(v)) + p(N_{G[N^o(v)]}^+[S'])\\
& \geq p(v) + \big(p(N^+(v))+p(N^-(v)\big)/2 + p(N^o(v))/2\\ 
& = p(V)/2.
\end{align*}
This proves the lemma.
\end{proof}

In the first proof of Theorem \ref{thm:frac_dom}, we
will use the following linear program.
Let $\cs$ be the set of all maximal stable sets in
$G$, and let $A$ be the matrix with $|V|$ rows and $|\cs|$ columns,
where for every $v\in V,S\in \cs$,
\begin{eqnarray}A(v,S)=
\left\{ \begin{array}{ll}
 1  &\text{\  if $v\in N^+[S]$,}\\
0 &\text{\ otherwise.} \end{array} \right. \label{eqn:comment1}
\end{eqnarray}
Let us consider the following linear program
$$
\text{(P)}\ \ \ \ \ \ \ \ \left. \begin{array}{ll}
 \text{Minimize }   &\textbf{1}^Tz\\
\text{Subject to} & Az\ge \textbf{1}\text{\ and\ } z\ge \textbf{0},\end{array} \right.$$
and its dual 
$$
\text{(D)}\ \ \ \ \ \ \ \ 
\left. \begin{array}{ll}
 \text{Maximize }   &\textbf{1}^Tw\\
\text{Subject to} & A^Tw\le \textbf{1}\text{\ and\ } w\ge \textbf{0}.\end{array} \right.$$

\begin{lemma}\label{lem:stable_lp}
The value of an optimal solution for (P) is at most $2$.
\end{lemma}

\begin{proof}
We prove that an optimal solution to (P) is at most 2, by proving that
the optimal solution to (D) is at most 2.  Then we apply the Strong
Duality Theorem~(see \cite{chong}, Theorem~17.2 for example) to
complete the proof.

Suppose for a contradiction that the optimal solution of (D) is
greater than 2. Then there is a function $w$ such that
$\textbf{1}^Tw >  2$ and $A^Tw\le\textbf{1}$. Then for every $S\in \cs$,
$\sum_{v\in V}A(v,S)w(v)\le 1$, and so by \eqref{eqn:comment1}, $$w(N^+[S])=\sum_{v\in
  N^+[S]}w(v) \le 1.$$
However, by Lemma~\ref{lem:matrix2}, there is
$S\in \cs$ such that $w(N^+[S]) \geq  w(V)/2 = (\textbf{1}^Tw)/2>1$.
This proves the lemma.\end{proof}

We now restate Theorem \ref{thm:frac_dom}.
\begin{theorem}\label{thm:frac dom2}
For any digraph $G=(V,E)$, we can construct a fractional dominating function $g:V
\to [0,1]$ such that $g(V) \leq 2 \alpha(G)$.
\end{theorem}

\begin{proof}[First proof of Theorem \ref{thm:frac dom2}]
Invoking Lemma \ref{lem:stable_lp}, there is a 
$z$ such that $Az\ge \textbf{1}$ and $\textbf{1}^Tz \leq
2$. Note that $z$ is a vector of length $|\cs|$ and $w$ is a vector of
length $|V|$. Let $g(v):= \sum_{S: v\in S}z(S)$ for every $v\in
V$. Let $\alpha = \alpha(G)$.

Note that for every $S\in \cs$, $|S|\le \alpha$  since $S$ is stable. We have
$$g(V)= \sum_{v\in V}\sum_{S: v\in S}z(S)=\sum_{S\in \cs}\sum_{v\in
  S}z(S)=\sum_{S\in \cs}|S|z(S)\le \alpha \sum_{S\in \cs}z(S)= \alpha
(\textbf{1}^Tz) \leq 2\alpha.$$

Fix $v$, since $Az\ge \textbf{1}$, we have $\sum_{S\in
  \cs}A(v,S)z(S)\ge 1$. In other words,
$$\sum_{S\in \cs:v\in N^+[S]}z(S)\ge 1.$$
Besides,
$$g(N^-[v])=\sum_{x:v\in N^+[x]}g(x)=\sum_{x:v\in N^+[x]}\sum_{S:x\in
  S}z(S)\ge \sum_{S\in \cs:v\in N^+[S]}z(S).$$ Thus, $g(N^-[v])\ge 1$
for every $v\in V$ (i.e., $g(v)$ is a fractional dominating function of
$G$). This proves the theorem.
\end{proof}

In the second proof of Theorem \ref{thm:frac dom2}, we will use the
following consequence of Farkas' Lemma (we refer the reader to
\cite{AL} for the proof of Lemma~\ref{lem:farkas}; see also Theorem 1
in \cite{fisher1996squaring}).
\begin{lemma}\label{lem:farkas}
For any digraph $G=(V,E)$, there exists a function $p:V\to [0,1]$ such that $p(V)=1$ and $p(N^-(v))\ge p(N^+(v))$ for every vertex $v$.
\end{lemma}

\begin{proof}[Second Proof of Theorem \ref{thm:frac dom2}]
We prove the theorem by induction on $\alpha(G)$. If $\alpha(G)=1$,
then $G$ is a tournament.  Let $p$ be a function satisfying
Lemma~\ref{lem:farkas}.  Let $g(v)=2p(v)$ for every $v \in V$. Then
$g(V)=2$ and for every $v$, we have
$$g(N^-[v])=2p(N^-[v]) \ge 2p(v) +p(N^-(v))+p( N^+(v)) \ge p(V) =1.$$
Thus, $g$ is a fractional domination function of $G$. We conclude that $\gamma^*(G)\le g(V)=  2=2\alpha(G)$ for the case $\alpha(G)=1$. 

If $\alpha(G)>1$, let $p$ be a function satisfying Lemma
\ref{lem:farkas}. By Observation~\ref{ob:indep}, we have
$\alpha(G[N^o(v)])\le \alpha(G)-1$ for every $v\in V$, and so
$\gamma^*(G[N^o(v)])\le 2(\alpha(G) - 1)$ by induction.  In the rest
of the proof, we write $G_v:=G[N^o(v)]$ for short.  For each vertex
$v$, let $g_v$ be a minimum fractional dominating function of
$G_v$. Set $g(x)=2p(x)+\sum_{y\in N^o(x)}g_y(x)p(y)$ for every vertex
$x$.  We show that $g$ is a fractional dominating function of $G$.
Note that $x\in N^o(y)$ if and only if $y\in N^o(x)$, and for every
$v,y$ with $v\in G_y$, $$\sum_{x\in N_{G_y}^-[v]}
g_y(x)=g_y(N_{G_y}^-[v])\ge 1$$ since $g_y$ is a fractional dominating
function of $G_y$.  Fix $v$, we have (we omit the subscript $G$ if
applicable)
\begin{align*}
g(N^-[v]) = \sum_{x\in N^-[v]}g(x) &=\sum_{x\in N^-[v]}\bigg( 2p(x) +\sum_{y\in N^o(x)} g_y(x)p(y)\bigg)\\
&=2\sum_{x\in N^-[v]}p(x) +\sum_{x\in N^-[v]}\sum_{y\in N^o(x)} g_y(x)p(y)\\
&=  2p(N^-[v])+\sum_{y\in V}p(y)\sum_{x\in N^-[v]\cap N^o(y)} g_y(x)\\
&\ge  2p(N^-[v])+\sum_{y\in N^o(v)}p(y)\sum_{x\in N^-[v]\cap N^o(y)} g_y(x)\\
&=  2p(N^-[v])+\sum_{y\in N^o(v)}p(y)\sum_{x\in N_{G_y}^-[v]} g_y(x)\\
&\ge  2p(N^-[v])+\sum_{y\in N^o(v)}p(y)\cdot 1\\
&=  2p(N^-[v])+p(N^o(v))\\
& \ge 2p(v) +p(N^-(v))+p(N^+(v))+p(N^o(v))\\
&\ge p(V),\\
&=1.
\end{align*}
We conclude that $g$ is a fractional dominating
function of $G$.

Note that $g_v$ is a minimum fractional dominating function of $g_v$, and so $g_v(N^o(v))=\gamma^*(G_v)\le 2\alpha(G_v)\le 2(\alpha(G)-1)$. Hence
\begin{align*}
g(V)=\sum_{v\in V}g(v) &=2\sum_{v\in V}p(v) + \sum_{v\in V}\sum_{y\in N^o(v)}g_y(v)p(y)\\
&=2+\sum_{y\in V}p(y)\sum_{v\in N^o(y)}g_y(v)\\
&\le 2+\sum_{y\in V}p(y)(2\alpha(G)-2)\\
&=2\alpha(G).
\end{align*}  
Thus, $\gamma^*(G)\le \sum_{v\in V}g(v)\le 2\alpha(G)$. This completes the proof.
\end{proof}


We can show that the bound in Theorem~\ref{thm:frac dom2} is sharp. 
\begin{proposition}
Given an arbitrarily small positive real number $\varepsilon$, for any positive integer $k$, there exists a digraph $G$ such that $\alpha(G)=k$ and $\gamma^*(G)>2k-\varepsilon$.
\end{proposition}
\begin{proof}
Let $r=\lceil 1/\varepsilon \rceil+1$. For $k=1$, let $G$ be an
Eulerian tournament with $V(G)=\{v_1,...,v_{2r-1}\}$ and
$N^+(v_i)=\{v_j:1\le j-i\mod (2r-1)\le r-1\}$. Let $g$ be a minimum
fractional dominating function of $G$. Suppose that
$g(v_r)=\min_{i}g(v_i)$. Note that $\sum_{i}g(v_i)=\gamma^*(G)\le 2$,
and $r$ is chosen sufficiently large so that $g(v_r)<2/(2r-1)<\varepsilon$. We have
\begin{align*}
\gamma^*(G)&=\sum_{i}g(v_i),\\
&=\sum_{1\le i\le r}g(v_i)+\sum_{r\le i\le 2r-1}g(v_i)-g(v_r)\\
&=\sum_{v_i\in N^-[v_r]}g(v_i)+\sum_{v_i\in N^-[v_{2r-1}]}g(v_i)-g(v_r)\\
&> 2-\varepsilon.
\end{align*}

For $k>1$, let $r=\lceil k/\varepsilon \rceil+1$.  Let $G$ be a
disjoint union of $k$ tournaments $G_1,\dots, G_k$, each is constructed
as in the case $k=1$. Since the $k$ tournaments are disjoint,
$\gamma^*(G)=\sum_{i}\gamma^*(G_i)\ge
k(2-\varepsilon/k)=2k-\varepsilon.$
\end{proof}

We can also show that almost surely a random tournament has a
fractional domination number close to the upper bound of 2.  First, we
need the following proposition.

\begin{proposition}
Let $G=(V,E)$ be a digraph of maximum out-degree $d$, then $\gamma^*(G)\ge n/(d+1)$.
\end{proposition}
\begin{proof}
Suppose that the statement was false, then there is a function $g:V\to
[0,1]$ such that $g(V)<n/(d+1)$ and $g(N^-[x])\ge 1$ for every $x$.
Then

\begin{align*}
n\le \sum_{x\in V}g(N^-[x]) &= g(V) + \sum_{x\in V}g(N^-(x))\\
& = g(V) + \sum_{ux\in E}g(u)\\
& = g(V)+ \sum_{u\in V}g(u)|N^+(u)| \\
&\leq g(V) + d\cdot g(V)\\
& = (d+1)g(V) < n,
\end{align*}
a contradiction.
\end{proof}

We also need Chernoff's inequality (see for example \cite{MR02}).

\begin{proposition}[Chernoff's Inequality]
Let $X$ be a binomial random variable consisting of $n$ Bernoulli trials, each with probability of success $p$.
Then, for all $0 < \epsilon < 1$, $$ Pr[|X- np| > \epsilon np] \leq 2 e^{-\epsilon^2 np /3}.$$ 
\end{proposition}

\begin{proposition} 
 For any $\varepsilon > 0$, $\mathbb{P}[\gamma^{*}(T_n) > 2 -
   \varepsilon] = 1-o(1)$ (i.e., $\gamma^{*}(T_n) > 2 - \varepsilon$ almost surely).
\end{proposition}

\begin{proof}
By Chernoff's bound, the probability
that a given vertex has out-degree more than $ n/2 + 10 \sqrt{n \log n}$ is 
$O(n^{-C})$ for some constant $C > 1$. Thus, the probability that there is a vertex
with out-degree more than $ n/2 + 10 \sqrt{n \log n}$ is, by the union bound, at most 
$n \cdot O(n^{-C}) = o(1)$. Thus, almost surely all the vertices
of $T_n$ have out-degree at most $ n/2 + 10 \sqrt{n \log n}$. 
It follows that almost surely  $\gamma^{*}(T_n) \geq \tfrac{n}{n/2 + 10 \sqrt{n \log n}}.$
\end{proof}


\section{Dominating sets in digraphs}\label{sec:dom-up}

In the previous section, we showed that the fractional domination
number of a digraph can be bounded from above by twice its
independence number.  In general, we cannot bound the (integral) domination
number of a digraph in terms of its independence number, as mentioned towards
the end of the introduction section.
Nevertheless, these two quantities can be related.

It is well known that a tournament has a dominating set of size at most
$\log{n}$~\cite{MV88}.  Analogously, we can show that a digraph
$G=(V,E)$ has a dominating set of size at most $\alpha(G) \cdot
\log{n}$.  For $S \subseteq V$, let $\chi(S)$ denote the chromatic
number of the underlying (undirected) graph of digraph $G[S]$.

\begin{lemma}\label{thm:bound1}
Every digraph $G=(V,E)$ with $n=|V|$ has a dominating set $D
\subseteq V$ such that $\chi(D) \leq \log{n}$.
\end{lemma}

\begin{proof}
We can assign each vertex a value $p(v) = 1$ and apply
Lemma \ref{lem:matrix2} to find a stable set $S$ such that $p(N^+[S])
\geq p(V)/2$.  We add the stable set $S$ to the dominating set and
recurse on the induced subgraph $G[V\setminus{N^+[S]}]$.  
Performing this routine $\log{n}$ times results in
the bound.
\end{proof}

Since each stable set $S$ has cardinality at most $\alpha(G)$, Lemma
\ref{thm:bound1} implies the following Theorem.
\begin{theorem}\label{thm:first}
Every digraph $G=(V,E)$ with $n = |V|$ has a dominating set of
size at most $\alpha(G) \cdot \log{n}$.
\end{theorem}

When is it possible to bound the domination number of a digraph purely
in terms of its independence number?  For example, Theorem
\ref{thm:first} implies that this can be done when the independence
number of a digraph is sufficiently large.

\begin{theorem}\label{thm:large-alpha1}
For every digraph $G=(V,E)$ with $n = |V|$, if $\alpha(G) \geq \log{n}$,
then $\gamma(G) \leq (\alpha(G))^2$.
\end{theorem}

Another case in which the domination number of a digraph can be bounded in terms of
its independence number is when the digraph is
directed-triangle-free.  For example, 
a directed-triangle-free digraph has independence number bounded by
$\alpha(G) \cdot \alpha(G)!$ (see Theorem 3 in \cite{GST12}).
Moreover, when
$\alpha(G) = 2$, this bound can be improved to $\gamma(G) \leq 3$ (see Theorem 4 in \cite{GST12}).

In \cite{HLNT17+}, we conjectured that the dichromatic number of a
directed-triangle-free digraph can be bounded as a polynomial function
of $\alpha(G)$.  Let $\vec \chi(G)$ denote the dichromatic number of a digraph.  Then for any digraph $G$, $\gamma(G)$
and $\vec \chi(G)$ can be related as follows.
\begin{observation}
For any digraph $G$, we have $\gamma(G) \leq \alpha(G) \cdot \vec \chi(G)$.
\end{observation}
This follows from the fact that in a legal coloring, each color class forms an induced
acyclic digraph and every acyclic digraph has a kernel (i.e., an
independent dominating set).  Thus, if the aforementioned conjecture
holds, then the following must also hold.

\begin{conjecture}\label{conj:P(alpha)}
There is an integer $\ell$ such that if $G$ is a
directed-triangle-free digraph with $\alpha(D) = \alpha$, 
then $\gamma(G) \leq \alpha^{\ell}$.
\end{conjecture}

Moreover, as pointed out in \cite{GST12}, the best theoretically possible upper bound on the
domination number of directed-triangle-free digraph $G$ cannot be better than $\gamma(G)
\leq \frac{3}{2}\alpha(G)$, as demonstrated by a disjoint union of
cyclically oriented pentagons.


\end{document}